\documentclass[12 pt]{amsart}

\usepackage{amsmath,amssymb,amsfonts}
\usepackage{amsthm}

 \addtocounter{section}{0}
\swapnumbers

\theoremstyle{plain}
\newtheorem{theorem}{Theorem}[section]
\newtheorem{proposition}[theorem]{Proposition}
\newtheorem{lemma}[theorem]{Lemma}

\theoremstyle{definition}
\newtheorem{definition and remark}[theorem]{Definition and Remark}

\newtheorem{notation, definition and remark}[theorem]{Notation, Definition and Remark}

\newtheorem{convention and remark}[theorem]{Convention and Remark}
\newtheorem{notation and remark}[theorem]{Notation and Remark}
\newtheorem{remark and definition}[theorem]{Remark and Definition}
\newtheorem{reminder and remark}[theorem]{Reminder and Remark}

\newtheorem{remark and problems}[theorem]{Remark and Problems}
\newtheorem{example and remark}[theorem]{Example and Remark}

\newcommand{\To}{\rightarrow}

\newcommand{\Nn}{\mathbb{N}_0}
\newcommand{\F}[1]{\frak #1}

\newcommand{\Z}{\mathbb{Z}}

\newcommand{\N}{\mathbb{N}}     % natural numbers
\begin{document}

\title{Some finiteness properties of Generalized graded local cohomology modules}
\author{ISMAEL AKRAY, ADIL KADIR JABBAR AND REZA SAZEEDEH}
\email{ ismaelakray@soran.edu.iq, adilqj@gmail.com,  rsazeedeh@ipm.ir }
%\date{Urmia, May 2004}



%%%%%%%%%%%%%%%%%%%%%%%%%%%%%%%%%%%%%%%%%%%%%%%%%%%%%%%%%%%%%%%%%%%

\begin{abstract}
Let $R = \bigoplus_{n \in \Nn} R_n$ be a Noetherian homogeneous ring
with local base ring $(R_0, \mathfrak{m}_0)$ and let $M$ and $N$ be
finitely generated graded $R$-modules. Let $i,j\in\mathbb{N}_0$. In
this paper we will study Artinianess of $\Gamma_{\frak
m_0R}(H_{R_+}^i(M,N)), H_{\frak m_0R}^1(H_{R_+}^i(M,N)),\\
H_{R_+}^i(M,N)/{\frak m_0}H_{R_+}^i(M,N), H_{R_+}^j(M,H_{\frak
m_0R}^i(N)), H_{\frak m_0R}^j(M,H_{R_+}^i(N))$, where $R_+$
denotes the irrelevant ideal of $R$.
\end{abstract}

\maketitle
\section{Introduction}\label{intro}

\hspace{0.4cm} Throughout this paper, let $R = \bigoplus_{n \in \Nn}
R_n$ be a Noetherian homogeneous ring with local base ring $(R_0,
\mathfrak{m}_0)$. So $R_0$ is a Noetherian ring and there are
finitely many elements $l_1, \dots, l_r \in R_1$ such that $R =
R_0[l_1, \dots, l_r]$. Let $R_+ := \bigoplus_{n \in \N} R_n$ denote
the irrelevant ideal of $R$ and let $\mathfrak{m} := \mathfrak{m}_0
\oplus R_+$ denote the graded maximal ideal of $R$. Finally let $M =
\bigoplus_{n \in \Z} M_n$ and $N = \bigoplus_{n \in \Z} N_n$ be
finitely generated graded $R$-modules.

\hspace{0.4cm} Herzog introduced a generalization of local
cohomology so called generalized local cohomology, denoted by
$H_{R_+}^i(M,N)$ which is isomorphic to
$\underset{\underset{n}{\rightarrow}}{\mbox{lim}}
{\rm{Ext}}_R^i(M/{{R_+}^n}M,N)$. We note that if $M=R$, then
$H_{R_+}^i(R,N)=H_{R_+}^i(N)$ is the usual local cohomology. As is
well known, the finiteness of the local cohomology modules have an
important role in commutative algebra and algebraic geometry. Many
of mathematicians work on finiteness of local cohomology. One of
approaches in finiteness is Artinianess. Authors, in [BFT],
[BRS], [RS] and [S] studied Artinianess of the graded modules
$\Gamma_{\frak m_0R}(H_{R_+}^i(N)$, $H_{R_+}^i(N)/{\frak
m_0}H_{R_+}^i(N)$ and $H_{\frak m_0R}^1(H_{R_+}^i(M))$. In the
recent paper we will study the Artinianess of Generalized graded
local cohomology. Here, we briefly mention some our results which
have been proved in this paper. Let $f$ be the least non-negative
integer such that the graded module $H_{R_+}^f(M,N)$ is not
finitely generated. We prove that $\Gamma_{\frak
m_0R}(H_{R_+}^i(M,N))$ is Artinian for each $i\leq f$. We also
prove that if $M$ is of finite projective dimension with ${\rm
pd}_RM=n$ and $c=c_{R_+}(N)$ is the largest non-negative integer
$i$ such that $H_{R_+}^i(N)\neq 0$, then
$H_{R_+}^{n+c}(M,N)/{\frak m_0}H_{R_+}^{n+c}(M,N)$ is Artinian.
In particular if $c_{R_+}(M,N)$ is the largest non-negative
integer $i$ such that $H_{R_+}^i(M,N)\neq 0$, then
$c_{R_+}(M,N)\leq {\rm pd}_RM+c_{R_+}(N)$. We prove a similar
result for invariant $a_{R_+}(N)$ which shows the largest
non-negative integer $i$ such that $H_{R_+}^i(N)$ is not
Artinain. Moreover, we show that if $H_{R_+}^i(M,N)$ is
$R_+$-cofinite, then $\Gamma_{\frak m_0R}(H_{R_+}^i(M,N))$ is
Artinian. Finally we study Artinianess of generalized local
cohomology when ${\rm dim}(R_0)\leq 1$. In this case, we prove
that $H_{R_+}^i(M,N)/{\frak m_0}H_{R_+}^i(M,N), \Gamma_{\frak
m_0R}(H_{R_+}^i(M,N)), H_{\frak m_0R}^1(H_{R_+}^i(M,N)),
H_{R_+}^j(M,H_{\frak m_0R}^1(N))$ and $H_{\frak
m_0R}^j(M,H_{R_+}^i(N))$ are Artinian for all $i, j\in
\mathbb{N}_0$.

\medskip
%###########################################################################
%$##########################################################################
\section{The results}\label{result}
\medskip

\hspace{0.4cm}It should be noted that local flat morphism of local
Noetherian rings is faithfully flat. So, if $R'_0$ is flat over
$R_0$ and $\F{m_0}'=\F{m_0}R'_0$, then $R'_0$ is faithfully flat
over $R_0$. Moreover, it follows from [K, Theorem 1] that if
$(R'_0,\F{m_0}')$ is a faithfully flat local $R_0$-algebra, then $A$
is a graded Artinian $R$-module if and only if
$A':=R'_0\otimes_{R_0}A$ is a graded Artinian module over
$R':=R'_0\otimes_{R_0}R$. In view of this argument, we have the
following proposition.

\medskip

\begin{proposition}\label{2.1}
Let $f:=f_{R_+}(M,N)$ be the least non-negative integer such that
$H_{R_+}^f(M,N)$ is not finitely generated.  Then $\Gamma_{\frak
m_0R}(H_{R_+}^i(M,N))$ is Artinian for each $i\leq f$.
\end{proposition}
\begin{proof}
We prove the assertion by induction on $i$. If $i=0$, then
$\Gamma_{R_+}(M,N)$ is finitely generated; and hence
$\Gamma_{\frak m_0R}(\Gamma_{R_+}(M,N))$ is finitely generated.
Thus there exists some positive integer $n$ such that ${\frak
m}^n\Gamma_{\frak m_0R}(\Gamma_{R_+}(M,N))=0$ and so the result
follows in this case. Now, suppose inductively that the result
has been proved for all values smaller than or equal to $i$ with
$i<f$ and we prove it for $i+1$, where $i+1\leq f$. Consider the
exact sequence $0\To \Gamma_{R_+}(N)\To N\To N/\Gamma_{R_+}(N)\To
0.$ Application of the functor $H_{R_+}^i(M,-)$ to the above exact
sequence induces the following exact sequence\\
$ H_{R_+}^{i+1} (M,\Gamma_{R_+}(N))\To
H_{R_+}^{i+1}(M,N)\stackrel{\phi}\To
H_{R_+}^{i+1}(M,N/\Gamma_{R_+}(N))\To
H_{R_+}^{i+2}(M,\Gamma_{R_+}(N))\hspace{0.5cm} \ddag.$ Set
$K_1={\rm Ker}\phi$ and $K_2={\rm Coker}\phi$.  We note that for
each $i$, there is an isomorphism
$H_{R_+}^i(M,\Gamma_{R_+}(N))\cong {\rm
Ext}^i(M,\Gamma_{R_+}(N))$ and then this module is an
$R_+$-torsion finitely generated graded $R$-module. Therefore
$K_1$ and $K_2$ are $R_+$-torsion finitely generated graded
$R$-modules. Consider $\Gamma_{\frak m_0R}(K_1)$ and
$\Gamma_{\frak m_0R}(K_2)$. By the previous argument these
modules are finitely generated and $R_+$-torsion and then there
exists some positive integer number $n$ such that ${\frak
m}^n\Gamma_{\frak m_0R}(K_1)={\frak m}^n\Gamma_{\frak
m_0R}(K_2)=0$ and this implies that $\Gamma_{\frak m_0}(K_1)$ and
$\Gamma_{\frak m_0R}(K_2)$ are Artinian. Now, in view of the
sequence $\ddag$, we can conclude that $\Gamma_{\frak
m_0}(H_{R_+}^{i+1}(M,N))$ is Artinian if and only if
$\Gamma_{\frak m_0}(H_{R_+}^{i+1}(M,N/\Gamma_{R_+}(N))$ is
Artinian. So we may assume that $\Gamma_{R_+}(N)=0$. Let
$\mathbf{x}$ be an indeterminate and let
$R'_0:=R_0[\mathbf{x}]_{\F{m_0}R_0[\mathbf{x}]}$,
$\F{m_0}':=\F{m_0}R'_0, R'=R'_0\otimes_{R_0}R$,
$M':=R'_0\otimes_{R_0}M$, and $N':=R'_0\otimes_{R_0}N$. Then by
the flat base change property of local cohomology, for each $i$
we have $R'_0\otimes_{R_0}\Gamma_{\frak
m_0R}(H_{R_+}^{i+1}(M,N))\cong \Gamma_{\frak
m_0R'_0}(H_{(R'_0\otimes_{R_0}R)_+}^{i+1}(R'_0\otimes_{R_0}M,R'_0\otimes_{R_0}N))$
and using the above argument $\Gamma_{\frak
m_0}(H_{R_+}^{i+1}(M,N))$ is Artinian if and only if
$\Gamma_{\frak
m_0R'_0}(H_{(R'_0\otimes_{R_0}R)_+}^{i+1}(R'_0\otimes_{R_0}M,R'_0\otimes_{R_0}N))$
is Artinian. So we may assume that the residue field $k$ of $R_0$
is infinite. As $\Gamma_{R_+}(N)=0$, there exists an element
$x\in R_1$ which is a non-zerodivisor with respect to $N$ and so
there is an exact sequence $0\To N(-1)\stackrel{x.}\To N\To
N/xN\To 0$ of graded $R$-modules. Application of the functor
$H_{R_+}^i(M,-)$ to this exact sequence induces the exact sequence
$$\dots \stackrel{x.}\To H_{R_+}^i(M,N)\To H_{R_+}^{i}(M,N/xN)\To
H_{R_+}^{i+1}(M,N)(-1)\stackrel{x.}\To H_{R_+}^i(M,N).$$ If we
consider $U^i:=H_{R_+}^i(M,N)/xH_{R_+}^i(M,N)$, then we have the
exact sequence $0\To U^i\To H_{R_+}^i(M,N/xN)\To
(0:_{H_{R_+}^{i+1}(M,N)}x)(-1)\To 0.$ Applying the functor
$\Gamma_{\frak m_0R}(-)$ to this sequence, we get the following
exact sequence $$\Gamma_{\frak m_0R}(H_{R_+}^i(M,N/xN))\To
\Gamma_{\frak m_0R}(0:_{H_{R_+}^i(M,N)}x)(-1)\To H_{\frak
m_0}(U^i)\hspace{0.5 cm}\maltese.$$ We note that $U^i$ is finitely
generated and $R_+$-torsion and so $H_{\frak m_0R}^i(U^i)\cong
H_{\frak m}^i(U^i)$ is Artinian. On the other hand, one can easily
show that $f(M,N/xN)\geq f(M,N)-1$ and so $i+1\leq f(M,N)$ implies
that $i\leq f(M,N)-1\leq f(M,N/xN)$; and hence by induction
hypothesis $\Gamma_{\frak m_0R}(H_{R_+}^i(M,N/xN))$ is Artinian.
Therefore by the sequence $\maltese$, the module $\Gamma_{\frak
m_0}(0:_{H_{R_+}^i(M,N)(-1)}x)=(0:_{\Gamma_{\frak
m_0}(H_{R_+}^i(M,N))(-1)}x)$ is Artinian. Now, since $\Gamma_{\frak
m_0}(H_{R_+}^i(M,N))$ is $x$-torsion, using Melkersson's Lemma this
module is Artinain.
\end{proof}

\medskip

\begin{lemma}\label{2.3}
Let $0\To M_1\To M\To M_2\To 0$ be a short exact sequence of
finitely generated graded $R$-modules. Then for any $R$-module
$N$, there is the following long exact sequence
$$\dots\To H_{R_+}^{i-1}(M_1,N)\To H_{R_+}^i(M_2,N)\To
H_{R_+}^i(M,N)\To H_{R_+}^i(M_1,N)\To\dots.$$
\end{lemma}
\begin{proof}
Let $\mathcal{I}:= 0\To I^0\To I^1\To I^2\To \dots$ be an
injective resolution of $N$. We note that for any finitely
generated $R$-module $M$ and any $i$, there is an isomorphism
$H_{R_+}^i(M,N)=H^i({\rm Hom}_R(M,\Gamma_{R_+}(\mathcal{I})))$
and by the basic properties of section functor each
$\Gamma_{R_+}(I^i)$ in $\Gamma_{R_+}(\mathcal{I})$ is injective.
Thus there is an exact sequence of complexes $$0\To {\rm
Hom}_R(M_2,\Gamma_{R_+}(\mathcal{I}))\To {\rm
Hom}_R(M,\Gamma_{R_+}(\mathcal{I}))\To {\rm
Hom}_R(M_1,\Gamma_{R_+}(\mathcal{I}))\To 0.$$ Now, by using a
basic theorem in homology theory, there is the following long
exact sequence of $R$-modules $$\dots\To H^{i-1}({\rm
Hom}_R(M_1,\Gamma_{R_+}(\mathcal{I})))\To H^i({\rm
Hom}_R(M_2,\Gamma_{R_+}(\mathcal{I})))$$$$\To H^i({\rm
Hom}_R(M,\Gamma_{R_+}(\mathcal{I})))\To H^i({\rm
Hom}_R(M_1,\Gamma_{R_+}(\mathcal{I})))\To\dots$$ and this
completes the proof.
\end{proof}

\begin{theorem}\label{2.4}
Let $M$ be of finite projective dimension with pd$_R(M)=n$ and
$c:=c_{R_+}(N)$ be the largest positive integer $i$ such that
$H_{R_+}^i(N)$ is not zero. Then the following condition hold.\\
 (i) The graded module $H_{R_+}^{n+c}(M,N)/\frak
m_0H_{R_+}^{n+c}(M,N)$ is Artinian.\\
 (ii) If $c_{R_+}(M,N)$ is the
largest positive integer $i$ such that $H_{R_+}^i(M,N)$ is not
zero, then $c_{R_+}(M,N)\leq {\rm pd}_RM+c_{R_+}(N)=n+c$.
\end{theorem}
\begin{proof}
(i) We proceed by induction on pd$_R(M)=n$. If $n=0$, the the
result is clear by [RS, Theorem 2.1]. Now, suppose inductively
that the result has been proved for all values smaller that $n>0$
and so we prove this for $n$. Since pd$_R(M)=n$, there exists a
positive integer $t$ and an exact sequence of graded modules $0\To
K\To R^t\To M\To 0$ such that pd$_R(K)=n-1$. In view of Lemma
\ref{2.3}, if we apply the functor $H_{R_+}^{n+c}(-,N)$ to this
exact sequence, we have the following exact sequence
$$H_{R_+}^{n+c-1}(K,N)\To H_{R_+}^{n+c}(M,N)\To H_{R_+}^{n+c}(R^t,N).$$
 We note that $n+c>c=c_{R_+}(N)$ and so $H_{R_+}^{n+c}(N)=0$. Now,
 application of the functor $R_0/{\frak m_0}\otimes_R-$ to the
 above exact sequence induces the following epimorphisms
$$H_{R_+}^{n+c-1}(K,N)/{\frak m_0}H_{R_+}^{n+c-1}(K,N)\To
 H_{R_+}^{n+c}(M,N)/{\frak m_0}H_{R_+}^{n+c}(M,N)\To 0.$$
 By using induction hypothesis, the graded module $H_{R_+}^{n+c-1}(K,N)/
 {\frak m_0}H_{R_+}^{n+c-1}(K,N)$ is Artinian. Thus the result
 follows by the above epimorphism. (ii) In this part, similar to (i),
 we can apply an easy induction on pd$_RM=n$.
 \end{proof}

\medskip

\medskip

\begin{theorem}\label{2.10}
Let $M$ be of finite projective dimension and let $a_{R_+}(M,N)$
be the largest non-negative integer $i$ such that $H_{R_+}^i(M,N)$
is not Artinian.
 Then we have the following conditions.\\
(i) $a_{R_+}(M,N)\leq{\rm pd}_RM+a_{R_+}(N)$, where $a_{R_+}(N)$
is the largest non-negative integer $i$ such that $H_{R_+}^i(N)$
is not Artinian.\\
(ii) $H_{R_+}^{a+n}(M,N)/\frak m_0H_{R_+}^{a+n}(M,N)$ is Artinian,
where $a=a_{R_+}(N)$ and  pd$_RM=n$.
\end{theorem}
\begin{proof}
(i) Let pd$_RM=n$ and $a=a_{R_+}(N)$. We prove the assertion by
induction on pd$_RM=n$. If $n=0$, the result is clear. Suppose,
inductively that $n>0$ and the result has been proved for all
values smaller than $n$ and so we prove it for $n$. As pd$_RM=n$,
there exists a positive integer $t$ and an exact sequence $0\To
M_1\To R^t\To M\To 0$ of $R$-modules such that pd$_RM_1=n-1$. In
view of Lemma \ref{2.3}, if we apply the functor $H_{R_+}^i(-,N)$
to the above exact sequence, we get the following exact sequence
of $R$-modules $H_{R_+}^{i-1}(M_1,N)\To H_{R_+}^i(M,N)\To
H_{R_+}^i(N)^t.$ Now, consider $i>a+n$. We note that
$i-1>a+n-1={\rm pd}_RM_1+a_{R_+}(N)$ and $i>a+n>a$. Thus, using
induction hypothesis, $H_{R_+}^{i-1}(M_1,N)$ and $H_{R_+}^i(N)$
are Artinian. Now, in view of the above exact sequence the result
follows. (ii) To prove this part, we again proceed by induction
on pd$_RM=n$. If $n=0$, then the result follows by [S, Theorem
2.3]. Now, suppose, inductively that $n>0$ and the result has
been proved for all values smaller than $n$ and so we prove it
for $n$. By a similar proof which mentioned in (i), there exists
an exact sequence of $R$-modules
$H_{R_+}^{a+n-1}(M_1,N)\stackrel{\alpha}\To
H_{R_+}^{a+n}(M,N)\stackrel{\beta}\To H_{R_+}^{a+n}(N)^t.$
Consider $X={\rm Im}(\alpha)$ and $Y={\rm Im}(\beta)$.  Since
$a+n>a$, the module $Y/\frak m_0 Y$ is Artinian. On the other
hand, by using induction hypothesis, one can easily see that
$X/\frak m_0X$ is Artinian. Now, the result follows easily.

\end{proof}

\begin{lemma}\label{2.4}
Let $i\in\mathbb{N}_0$ and $H_{R_+}^i(M,N)$ is $R_+$-cofinite. Then
$\Gamma_{\frak m_0R}(H_{R_+}^i(M,N))$ is Artinian.
\end{lemma}
\begin{proof}
Since $H_{R_+}^i(M,N)$ is $R_+$-cofinite, ${\rm
Hom}_R(R/R_+,H_{R_+}^i(M,N))$ is finitely generated and
$R_+$-torsion. Thus $\Gamma_{\frak m_0R}({\rm
Hom}_R(R/R_+,H_{R_+}^i(M,N)))\cong\Gamma_{\frak
m_0R}((0:_{H_{R_+}^i(M,N)}R_+))=(0:_{\Gamma_{\frak
m_0}(H_{R_+}^i(M,N))}R_+)$ is finitely generated and $\frak
m$-torsion. It implies that the last term is Artinian. On the other
hand, since $\Gamma_{\frak m_0R}(H_{R_+}^i(M,N))$ is $R_+$-torsion,
it is Artinian.
\end{proof}
\medskip

\begin{proposition}\label{2.5}
Let $M$ be of finite projective dimension and let $R_+$ be a
principal graded ideal of $R$. Then $\Gamma_{\frak
m_0}(H_{R_+}^i(M,N))$ is Artinian for each $i$.
\end{proposition}
\begin{proof}
As $R_+$ is principal, using [DS, Theorem 2.8], $H_{R_+}^i(M,N)$ is
$R_+$-cofinite for each $i$. Now, the assertion follows by the
previous lemma.
\end{proof}

\medskip

\begin{lemma}\label{2.2}
Let $M$ and $N$ be finitely generated graded $R$-modules and $N$
be $\frak m_0$-torsion. Then $H_{R_+}^i(M,N)$ is Artinian for all
$i$.
\end{lemma}
\begin{proof}
Since $N$ is $\frak m_0$-torsion, there exists an injective
resolution $\mathcal{I}:= 0\To I^0\To I^1\To\dots$ such that each
$I^i$ is $\frak m_0$-torsion.  By the definition of generalized
local cohomology there are the
following isomorphisms\\
$H_{R_+}^i(M,N)\cong H^i(\Gamma_{R_+}(M,\mathcal{I}))\cong
H^i(\Gamma_{R_+}(M,\Gamma_{\frak m_0R}( \mathcal{I})))\cong
H^i({\rm Hom}(M,\Gamma_{\frak m}(\mathcal{I}))\cong H_{\frak
m}^i(M,N).$ By the basic properties of generalized local
cohomology, the last term is Artinian; and hence the result
follows.
\end{proof}

\medskip

\begin{proposition}\label{2.6}
Let ${\rm dim}(R_0)\leq 1$. Then for every $i\in\mathbb{N}_0$, the
module $H_{R_+}^i(M,N)/\\\frak m_0H_{R_+}^i(M,N)$ is Artinian
\end{proposition}
\begin{proof}
If ${\rm dim}(R_0)=0$, then $N$ is $\frak m_0$-torsion and so in
view of Lemma \ref{2.2}, the graded module $H_{R_+}^i(M,N)$ is
Artinian for each $i$. By using [BFT, Lemma 2.2 ] we can get the
assertion. Now, suppose that dim$(R_0)=1$ and consider the short
exact sequence $0\To \Gamma_{\frak m_0R}(N)\To N\To
N/\Gamma_{\frak m_0R}(N)\To 0$ of graded $R$-modules. Application
of the functor $H_{R_+}^i(M,-)$ to this exact sequence induces the
following exact sequence $H_{R_+}^{i}(M,\Gamma_{\frak m_0R}(N))\To
H_{R_+}^i(M,N)\To H_{R_+}^i(M,N/\Gamma_{\frak m_0R}(N))\To
H_{R_+}^{i+1}(M,\Gamma_{\frak m_0R}(N)).$ In view of Lemma
\ref{2.2} and using [BFT, Lemma 2.2], one can easily show that
for each $i$, the module $R_0/{\frak
m_0}\otimes_{R_0}H_{R_+}^i(M,N)$ is Artinian if and only if
$R_0/{\frak m_0}\otimes_{R_0}H_{R_+}^i(M,N/\Gamma_{\frak
m_0}(N))$ is Artinian. So we may assume that $\Gamma_{\frak
m_0R}(N)=0$. Now, this fact implies that there exists an element
$x\in\frak m_0$ which is a non-zerodivisor of $M$ and then there
exists a short exact sequence $0\To N\stackrel{x.}\To N\To
N/xN\To 0$ of graded $R$-modules. Application of the functor
$H_{R_+}^i(M,-)$ to the above exact sequence induces the
following exact sequence
$$H_{R_+}^i(M,N)\stackrel{x.}\To H_{R_+}^i(M,N)\To
H_{R_+}^i(M,N/xN)\To H_{R_+}^{i+1}(M,N).$$ Since $R_0$ is of
dimension one, $N/xN$ is $\frak m_0$-torsion and so Lemma
\ref{2.2} implies that $H_{R_+}^i(M,N/xN)$ is Artinian for each
$i$. Now, using [BFT, Lemma 2.2], $R_0/\frak
m_0\otimes_{R_0}H_{R_+}^i(M,N)/xH_{R_+}^i(M,N)$ is Artinian. On
the other hand, application of the functor $R_0/\frak
m_0\otimes_{R_0}-$ to the above long exact sequence implies the
following isomorphism $R_0/\frak
m_0\otimes_{R_0}H_{R_+}^i(M,N)\cong R_0/\frak
m_0\otimes_{R_0}H_{R_+}^i(M,N)/xH_{R_+}^i(M,N)$; and hence the
assertion follows.
\end{proof}

\medskip

\begin{proposition}\label{2.8}
Let dim$(R_0)\leq 1$. Then we have the following conditions.\\
(i) The graded $R$-module $\Gamma_{\frak m_0R}(H_{R_+}^i(M,N))$ is
Artinian for each $i\in\mathbb{N}_0$.\\
 (ii) The graded $R$-module $H_{\frak m_0R}^1(H_{R_+}^i(M,N))$ is
Artinian for each $i\in \mathbb{N}_0$.
\end{proposition}
\begin{proof}
If dim$(R_0)=0$, then $R_0$ is Artinian. In this case any finitely
generated $R$-module is Artinian, and then for each $i$, the
module $H_{R_+}^i(M,N)$ is Artinian. Thus $\Gamma_{\frak
m_0R}(H_{\frak m_0}^1(H_{R_+}^i(M,N)))$ is Artinian and $H_{\frak
m_0}^1(H_{R_+}^i(M,N))=0$ for each $i$ and so (i) and (ii) are
clear in this case. Now, assume that dim$(R_0)=1$. (i). By
applying the functor $H_{R_+}^i(M,-)$ to the exact sequence $0\To
\Gamma_{\frak m_0R}(N)\To N\To N/\Gamma_{\frak m_0R}(N)\To 0$,
and applying the functor $\Gamma_{\frak m_0R}(-)$ to the induced
exact functor, we can conclude that $\Gamma_{\frak
m_0R}(H_{R_+}^i(M,N))$ is Artinian if and only if $\Gamma_{\frak
m_0R}(H_{R_+}^i(M,N/\Gamma_{\frak m_0R}(N)))$ is Artinian and so
we may assume that $\Gamma_{\frak m_0R}(N)=0$. Now, let
$x\in\frak m_0$ be a non-zerodivisor of $N$. Then there is an
exact sequence $0\To N\stackrel{x.}\To N\To N/xN\To 0$ of graded
$R$-modules. Application of the functor $H_{R_+}^i(M,-)$ to this
sequence induces the following exact sequence
$$H_{R_+}^{i-1}(M,N/xN)\To H_{R_+}^i(M,N)\stackrel{x.}\To
H_{R_+}^i(M,N).$$ We note that $N/xN$ is $\frak m_0$-torsion and
so by Lemma \ref{2.2}, $H_{R_+}^{i-1}(M,N/xN)$ is Artinian and
then $(0:_{H_{R_+}^i(M,N)}x)$ is Artinian. This implies that
$\Gamma_{\frak m_0}((0:_{H_{R_+}^i(M,N)}x)=(0:_{\Gamma_{\frak
m_0R}(H_{R_+}^i(M,N))}x)$ is Artinian. Now, since $\Gamma_{\frak
m_0R}(H_{R_+}^i(M,N))$ is $x$-torsion, by using Melkesson's Lemma,
it is Artinian. (ii). We proceed the assertion by induction on
$i$. If $i=0$, then $H_{\frak m_0R}^1(\Gamma_{R_+}(M,N))=H_{\frak
m}^1(\Gamma_{R_+}(M,N)).$ We note that the last term is Artinian
because $\Gamma_{R_+}(M,N)$ is finitely generated. Suppose,
inductively that the result has been proved for all values smaller
than $i$ and so we prove it for $i$. Let $y\in\frak m_0$ be a
system of parameter of $\frak m_0$. As $M$ is finitely generated,
for some positive integer $t$ there exists a short exact sequence
$0\To K\To R^t\To M\To 0$ of $R$-modules. In view of Lemma
\ref{2.3}, if we apply the functor $H_{R_+}^i(-,N)$ to the above
exact sequence, we get the following exact sequence
$$H_{R_+}^{i-1}(K,N)\stackrel{\alpha}\To H_{R_+}^i(M,N)\To
H_{R_+}^i(R^t,N)\stackrel{\beta}\To H_{R_+}^i(K,N).$$ Consider
$A:={\rm Im}(\alpha)$, $B:={\rm Ker}(\beta)$ and $C:={\rm
Im}(\beta)$. Application of the functor $H_{yR}^i$ to the above
exact sequence gives the epimorphism
$H_{yR}^1(H_{R_+}^{i-1}(K,N))\twoheadrightarrow H_{yR}^1(A)$, the
monomorphism $\Gamma_{yR}(C)\rightarrowtail
\Gamma_{yR}(H_{R_+}^i(K,N))$, and the exact sequence
$\Gamma_{yR}(C)\To H_{yR}^1(B)\To
H_{yR}^1(H_{R_+}^i(R^t,N))\hspace{0.3cm} (\ddag).$ By using
induction hypotheses $H_{yR}^1(H_{R_+}^{i-1}(K,N))=H_{\frak m_0R
}^1(H_{R_+}^{i-1}(K,N))$ is Artinian and so is $H_{yR}^1(A)$. On
the other hand since, by (i), the module
$\Gamma_{yR}(H_{R_+}^i(K,N))$ is Artinain, the module
$\Gamma_{yR}(C)$ is Artinian. Using [BFT, Theorem 2.5], the
module $H_{yR}^1(H_{R_+}^i(R^t,N))$ is Artinian and then the
exact sequence $\ddag$ and the previous arguments imply that
$H_{yR}^1(B)$ is Artinian. Now, since both $H_{yR}^1(A)$ and
$H_{yR}^1(B)$ are Artinian, one can easily deduce that
$H_{yR}^1(H_{R_+}^i(M,N))=H_{\frak m_0R}^1(H_{R_+}^i(M,N))$ is
Artinian.
\end{proof}

\medskip

\begin{proposition}\label{2.9}
Let dim$(R_0)\leq 1$. Then $H_{R_+}^p(M,H_{\frak m_0R}^1(N))$ is
Artinian for each $p\in \mathbb{N}_0$.
\end{proposition}
\begin{proof}
If dim$(R_0)=0$, then $H_{\frak m_0R}^1(N)=0$ and so the result is
clear in this case. Now, assume that dim$(R_0)=1$. By the
Grothendieck spectral sequence (see [R, Theorem 11.38]), for each
$p,q\in\mathbb{N}_0$, there is
$$E_2^{p, q}:=H_{R_+}^p(M,H_{\frak m_0R}^q(N))\underset{p}\Longrightarrow H_{\frak
m}^{p+q}(M,N).$$ As dim$(R_0)=1$, we have $H_{\frak m_0R}^q(N)=0$
for all $q>1$ and then $E_2^{p, q}=0$ for all $q\neq 0,1$. Thus we
can apply the dual of [W, Ex. 5.2.2] to get the following exact
sequence $$E_2^{p+1, 0}\To H_{\frak m}^{p+1}(M,N)\To E_2^{p, 1}\To
E_2^{p+2,0}\To H_{\frak m}^{p+2}(M,N).$$ It is easy to see that
$H_{\frak m}^{p+1}(M,N)$ and
$E_2^{p+2,0}=H_{R_+}^{p+2}(M,\Gamma_{\frak m_0R}(N))=H_{\frak m
}^{p+2}(M,\Gamma_{\frak m_0R}(N))$ are Artinian. Thus the above
exact sequence implies that $E_2^{p,1}=H_{R_+}^p(M,H_{\frak
m_0R}^1(N))$ is Artinian.
\end{proof}

\medskip

\begin{proposition}\label{2.10}
Let dim$(R_0)\leq 1$. Then $H_{\frak m_0R}^j(M,H_{R_+}^i(N))$ is
Artinian for each $j,i\in \mathbb{N}_0$.
\end{proposition}
\begin{proof}
If dim$(R_0)=0$, then each finitely generated $R$-module is $\frak
m_0$-torsion. Thus $H_{R_+}^i(N)$ is Artinian and so is $\frak
m_0$-torsion. Then for each $j$, there is an isomorphism $H_{\frak
m_0R}^j(M,H_{R_+}^i(N))\cong {\rm Ext}_R^j(M,H_{R_+}^i(N))$. One
can easily show that the last module is Artinian. Now, assume that
dim$(R_0)=1$. We proceed by induction on $j$. If $j=0$, then we
have $H_{\frak m_0R}^0(M,H_{R_+}^i(N))={\rm Hom}_R(M,\Gamma_{\frak
m_0R}(H_{R_+}^i(N)))$. By using [BFT, Theorem 2.5], the module
$\Gamma_{\frak m_0R}(H_{R_+}^i(N))$ is Artinian and so one can
easily show that ${\rm Hom}_R(M,\Gamma_{\frak
m_0R}(H_{R_+}^i(N)))$ is Artinian. Now, we assume that $j>0$ and
the result has been proved for all values smaller than $j$ and we
prove it for $j$. Since $M$ is finitely generated, for some
positive integer $t$, there is an exact sequence $0\To K\To
R^t\To M\To 0$ of $R$-module. In view of Lemma \ref{2.3}, if we
apply the functor $H_{\frak m_0R}^j(-,H_{R_+}^i(N))$ to the above
exact sequence, we have the following exact sequence
$$H_{\frak m_0R}^{j-1}(K,H_{R_+}^i(N))\To H_{\frak
m_0R}^j(M,H_{R_+}^i(N))\To H_{\frak m_0R}^j(R^t,H_{R_+}^i(N)).$$
We note that by induction hypotheses, the module $H_{\frak
m_0R}^{j-1}(K,H_{R_+}^i(N))$ is Artinian and $H_{\frak
m_0R}^j(R^t,H_{R_+}^i(N))=0$ for all $j>1$ and also $H_{\frak
m_0R}^1(R^t,H_{R_+}^i(N))$ is Artinian by [BFT, Theorem 2.5]. Now,
in view of the above exact sequence, we get our assertion.
\end{proof}

%%%%%%%%%%%%%%%%%%%%%%%%%%%%%%%%%%%%%%%%%%%%%%%%%%%%%%%%%%%%%%%%%%%

\bibliographystyle{plain}

\begin{thebibliography}{9999999}

%\bibitem[BS]{BS} M. Brodmann and R. Y. Sharp, {\em Local cohomology:
%an algebraic introduction with geometric applications}, Cambridge
%Studies in Advanced Mathematics 60, Cambridge University Press
%(1998).

\bibitem[BFT]{BFT} M. Brodmann, S. Fumasoli and R. Tajarod, {\em Local cohomology over
homogeneous rings with one-dimensional local base ring}, Proc.
Amer. Math. Soc 131 (2003), 2977 - 2985.

\bibitem[BRS]{BRS} M. Brodmann, F. Rohrer and R. Sazeedeh, {\em
Multiplicities of graded components of local cohomology modules},
J. Pure. Appl. Algebra 197(2005), 249-278.

\bibitem[DS]{DS} K. Divaani-Aazar and R. Sazeedeh, {\em Cofiniteness of Generalized local
cohomology modules}, Colloq. Math (2) 99 (2004), 283-290.

\bibitem[K]{K} D. Kirby, {\em Artinian modules and Hilbert polynomials},
 Quart. J. Math. Oxford (2) 24 (1973), 47 - 57.

\bibitem[KS]{KS} M. Katzman and R. Y. Sharp, {\em Some properties
of top graded local cohomology modules}, J. Algebra  259 (2003),
599 - 612.

\bibitem[R]{R} J. J. Rotman, {\em An Introducion to
Homological Algebra}; New York: Acad Press, 1979.

\bibitem[RS]{RS} C. Rotthaus and L. M. \c{S}ega, {\em Some
properties of graded local cohomology modules}, J. Algebra (2004).

\bibitem[S]{S} R. Sazeedeh, {\em Artinianess of graded local
cohomology modules}, Proc. Amer. Math. Soc, to appear.

\bibitem[W]{W} C. A. Weibel, {\em An introduction to homological algebra},
Camb.Univ. Press, 1994.
\end{thebibliography}

%\vskip 1cm \footnotesize {Ismael Akray) {\sc Department of Mathematics, Soran University, Erbil, Iraq }

\medskip
%{\it E-mail address}: {\tt ismaelakray@soran.edu.iq}

%\vskip 1cm \footnotesize {ADIL KADIR JABBAR) {\sc Department of Mathematics, Sulaimani University, Sulaimani, Iraq }

\medskip

%{\it E-mail address}: {\tt  adilqj@gmail.com}

%\vskip 1cm \footnotesize {REZA SAZEEDEH) {\sc Department of Mathematics, Urmia University, Urmia, Iran -and-\\ Research %Institute for Fundamental Sciences, Tabriz, Iran}

\medskip

%{\it E-mail address}: {\tt  rsazeedeh@ipm.ir}

\end{document}